\documentclass[12pt]{amsart}
\usepackage{a4wide}

\usepackage{amsfonts,amsthm,amsmath,amssymb,amscd}
\usepackage{graphicx}

\theoremstyle{plain}
\newtheorem{lem}{Lemma}
\newtheorem{prop}[lem]{Proposition}

\newtheorem*{thm*}{Theorem}
\newtheorem{cor}[lem]{Corollary}
\newtheorem*{cor*}{Corollary}

\newtheorem*{thmA}{Theorem A}
\newtheorem*{thmB}{Theorem B}
\newtheorem*{corA'}{Corollary A'}

\theoremstyle{definition}

\newtheorem*{defn*}{Definition}

\newtheorem*{ex*}{Example}
\newtheorem{rem}[lem]{Remark}
\newtheorem*{rem*}{Remark}

\theoremstyle{remark}

\DeclareMathOperator{\diam}{diam}
\DeclareMathOperator{\dist}{dist}
\DeclareMathOperator{\Arg}{Arg}

\DeclareMathOperator{\dim_H}{dim_H}

\DeclareMathOperator{\length}{length}
\newcommand{\C}{\mathbb C}

\newcommand{\R}{\mathbb R}
\newcommand{\Z}{\mathbb Z}

\newcommand{\G}{\mathcal G}
\newcommand{\HH}{\mathcal H}
\newcommand{\KK}{\mathcal K}

\newcommand{\QQ}{\mathcal Q}

\newcommand{\LL}{\mathcal L}

\newcommand{\bd}{\partial}

\renewcommand{\Re}{\textup{Re}}
\renewcommand{\Im}{\textup{Im}}

\begin{document}

\title{Dimension properties of the boundaries of exponential basins}

\date{February 6, 2009}

\author{Krzysztof Bara\'nski}
\address{Institute of Mathematics, University of Warsaw,
ul.~Banacha~2, 02-097 Warszawa, Poland}
\email{baranski@mimuw.edu.pl}

\author{Bogus{\l}awa Karpi\'nska}
\address{Faculty of Mathematics and Information Science, Warsaw
University of Technology, Pl.~Politechniki~1, 00-661 Warszawa, Poland}
\email{bkarpin@mini.pw.edu.pl}

\author{Anna Zdunik}
\address{Institute of Mathematics, University of Warsaw,
ul.~Banacha~2, 02-097 Warszawa, Poland}
\email{A.Zdunik@mimuw.edu.pl}


\thanks{Research supported by Polish MNiSW Grant N N201 0234 33 and
EU FP6 Marie Curie Programmes RTN CODY and ToK SPADE2. The second author is
supported by Polish PW Grant 504G 1120 0011 000.}

\begin{abstract} We prove that the boundary of a component $U$ of the
  basin of an attracting periodic cycle (of period greater than $1$) for an
  exponential map on the complex plane has Hausdorff 
  dimension greater than $1$ and less than $2$. Moreover,
the set of points in the boundary of $U$ which do not escape to infinity
has Hausdorff dimension (in fact: hyperbolic dimension) greater than
$1$, while the set of points in
  the boundary of $U$ which escape to infinity has Hausdorff dimension $1$.
\end{abstract}


\maketitle

\section{Introduction}\label{sec:intro}

Let
\[
f(z) = \lambda \exp(z)
\]
for $\lambda \in \C \setminus \{0\}$, $z \in \C$  be an exponential map with an
attracting periodic cycle of period $p$. Then the Julia set $J(f)$ is
equal to the boundary of the (entire) basin $B$ of attraction to this
cycle. We are interested in dimension properties of this
boundary. We consider also the escaping set
\[
I(f) = \{z\in \C: f^{n}(z) \to \infty \text{ as } n \to \infty\}
\]
and the set of points in the Julia set with bounded forward
trajectories
\[
J_\text{bd}(f) = \{z \in J(f): \{f^n(z)\}_{n = 0}^\infty \text{ is bounded}\}.
\]
In our context, both set are contained in the boundary of $B$. 

In the well-known case $p = 1$, the attracting basin consists of a
unique simply connected component $U$. As proved in \cite{McM}, the
Hausdorff dimension $\dim_H$ of its boundary $\bd U = J(f)$ is then equal to
$2$. In fact, it follows from the proof in \cite{McM} that $\dim_H(\bd
U) = \dim_H(\bd U \cap I(f)) = 2$. On the other hand, we have $1 < \dim_H(\bd U
\cap J_\text{bd}(f)) \leq \dim_H(\bd U
\setminus I(f)) < 2$ (see \cite{K1, UZfiner}). A refined analysis
of the dimension of subsets of $I(f)$ according to the speed of convergence to
$\infty$ can be found in \cite{KU}. Hence, it follows from the above
results that in the case $p = 1$ the dimension of $\bd U$ is carried
by the set of escaping points.

Consider now the case $p > 1$. Then the basin $B$ of the attracting cycle 
consists of infinitely many disjoint simply connected components. Again,  
the Hausdorff dimension of the boundary of $B$ is equal to $2$ (see
\cite{McM}). However, this does not imply that 
the boundaries of components of $B$ have Hausdorff
dimension $2$. In fact, in this paper we prove the following theorem.  

\begin{thmA} Let $p > 1$ and let $U$ be a component of the basin of
the attracting cycle for the map $f$. Then the boundary of $U$ has Hausdorff
dimension greater than $1$ and less than $2$.
\end{thmA}

This shows that for $p > 1$ the dimension of $\bd U$ is carried by
points in the Julia set, which are not in the boundary of any component
of the Fatou set (so-called buried points). By Theorem~A, we
immediately obtain:

\begin{corA'} If $p > 1$, then the set of buried points for the map
$f$ has Hausdorff dimension $2$.
\end{corA'}

Surprisingly, for $p > 1$ the role of escaping and non-escaping points
in $\bd U$ from dimensional point of view is quite different than in
the case $p = 1$. This is shown in the second result proved in this paper. 

\begin{thmB} If $p > 1$ and $U$ is a component of the basin of
the attracting cycle for the map $f$, then
\[
1 < \dim_H(\bd U \cap J_\text{\rm bd}(f)) \leq \dim_H(\bd U
\setminus I(f)) =  \dim_H(\bd U) < 2,
\]
while 
\[
\dim_H(\bd U \cap I(f)) = 1.
\]
\end{thmB}

\begin{rem*}
In fact, the proof of Theorem~B shows that $\bd U \cap J_\text{bd}(f)$
has hyperbolic dimension greater than $1$. 
\end{rem*}

Therefore, in the case $p > 1$ the dimension of $\bd U$ is carried by
the set of non-escaping points. 

\begin{rem*} It is easy to check that for $z \in \bd U$ we have
\[
f^{n}(z) \xrightarrow[n \to \infty]{} \infty \iff f^{np}(z)
\xrightarrow[n \to \infty]{} \infty
\]
(see Lemma~\ref{lem:Re}).
\end{rem*}

Recall that $J(f)$ consists of
so-called hairs homeomorphic to $[0, \infty)$, tending to $\infty$,
composed of points with 
given symbolic itineraries. All points from a hair outside its
endpoint (the point corresponding to $0$ in $[0, \infty)$) are
contained in $I(f)$. In the case $p = 1$ the hairs are pairwise
disjoint and the Julia set is homeomorphic to a so-called straight
brush (see \cite{AO}), while in the case $p > 1$ some hairs have
common endpoints and the Julia set is a modified straight
brush (see \cite{BDLKM}). In both cases the Hausdorff dimension of
the union of hairs without endpoints is equal to $1$ (see \cite{K2,
SZ}). Hence, the hairs without endpoints are insignificant from
dimensional point of view.

The plan of the paper is as follows. In Section~\ref{sec:fingers} we
describe combinatorics of exponential maps with a periodic basin of
attraction and provide a symbolic description of the boundaries of its
components (Proposition~\ref{kody}). Then, in Section~\ref{sec:esc} we
prove that the set of 
escaping points in the boundary $U$ has Hausdorff dimension $1$. In
Section~\ref{sec:non-esc} it is shown that the set of points in the
boundary of $U$ with
bounded trajectories has hyperbolic dimension 
greater than $1$. Finally, Section~\ref{sec:concl} concludes the proof.  

\section{Combinatorics of exponentials with attracting periodic
  basins}\label{sec:fingers} 

The combinatorics of the Julia sets of exponential maps with an
attracting periodic cycle of period $p > 1$ was described in \cite{BD,
  BDLKM}. For completeness, we 
briefly recall the construction from these papers.

Let $z_0, \ldots, z_p = z_0$ be the attracting cycle and let $U_0,
\ldots, U_p = U_0$ be the components of the basin of attraction, such that $z_j
\in U_j$. Of course, we have $f^p(U_j) \subset U_j$. All $U_j$ are
simply connected (see e.g.~\cite{Ba}). Since 
$0$ is the only singular value of $f$, it must be in a 
component containing some $z_j$, so we can assume $0 \in
U_1$. It is easy to find a simply connected domain $B_{p + 1} \subset
U_1$ such that $0, z_1 \in B_{p + 1}$, $\bd B_{p + 1}$ is a Jordan
curve in $U_1$ and $\overline{f^p(B_{p + 1})} \subset B_{p + 1}$.
Let $B_p = f^{-1}(B_{p + 1})$. Then $B_p$ is simply
connected, contains some left-half plane and $\bd B_p$ is a
curve homeomorphic to the real line, 
contained in $\{z \in \C: a < \Re(z) < b\}$ for some $a, b \in
\R$. This curve is periodic of period $2\pi i$ and its two ``ends''
have imaginary parts tending respectively to $\pm \infty$.  

For $j = p - 1, \ldots, 0$ define successively $B_j$
to be the component of $f^{-1}(B_{j + 1})$ containing $z_j$. The
construction implies that $f$ maps univalently $B_j$ onto $B_{j
  + 1}$ for $j = 1, \ldots, p - 1$, such that $B_1 \supset
\overline{B_{p + 1}}$, and $f$ on $B_0 = f^{-1}(B_1)$ is a universal covering
of $B_1 \setminus \{0\}$. 

Let $C_p \subset B_p$ be a domain
containing $z_0$ and some left half-plane $\{z: \Re(z) < c\}$, such
that $\bd C_p$ is a curve homeomorphic to the 
real line, its two ``ends'' have imaginary parts tending to $\pm
\infty$, moreover $\bd C_p$ coincides with $\bd B_p$ on $\{z: |\Im(z)|
< L_1\}$ and is vertical on $\{z: |\Im(z)| > L_2\}$ for some large constants
$L_1, L_2 > 0$.    

Define successively $C_j \subset B_j$ to be the component of
$f^{-1}(C_{j + 1})$ 
containing $z_j$ for $j = p - 1, \ldots, 0$. Again, $f$ maps
univalently $C_j$ onto $C_{j + 1}$ and $f$ on $C_0 = f^{-1}(C_1)$ is a
universal 
covering of $C_1 \setminus \{0\}$. Moreover, since
$\overline{B_{p + 1}}$ is compact, we have $\overline{B_{p + 1}}
\subset C_1$ if the constant 
$L_1$ is chosen sufficiently large. This implies
\[
\overline{C_j} \subset U_j \quad \text{for } j = 0, \ldots, p -
1 \qquad \text{and} \qquad \overline{C_p} \subset C_0.
\]
In particular, $\overline{C_0}, \ldots, \overline{C_{p - 1}}$ are
pairwise disjoint. Since $C_0$ contains some 
left half-plane and $0 \notin C_0$, the set $\C \setminus C_0$ consists of
infinitely many disjoint simply connected, closed, connected sets
$\HH_s$, $s \in \Z$, such that 
\begin{equation} \label{eq:H_s}
\HH_s = \HH_0 + 2 \pi i s
\end{equation}
Similarly, for $j = 1, \ldots, p - 1$ the set  
$\overline C_j$ is a simply connected, closed, connected subset of
$\bigcup_s\HH_s$, such that
\begin{equation}\label{eq:k_j}
z_j \in C_j \subset U_j \subset \HH_{k_j} \qquad \text{for } j = 1,
\ldots p - 1,
\end{equation}
for some $k_j \in \Z$. See Figure~\ref{fig:c}.

\begin{figure}[!ht]
\begin{center}
\includegraphics*[height=10cm]{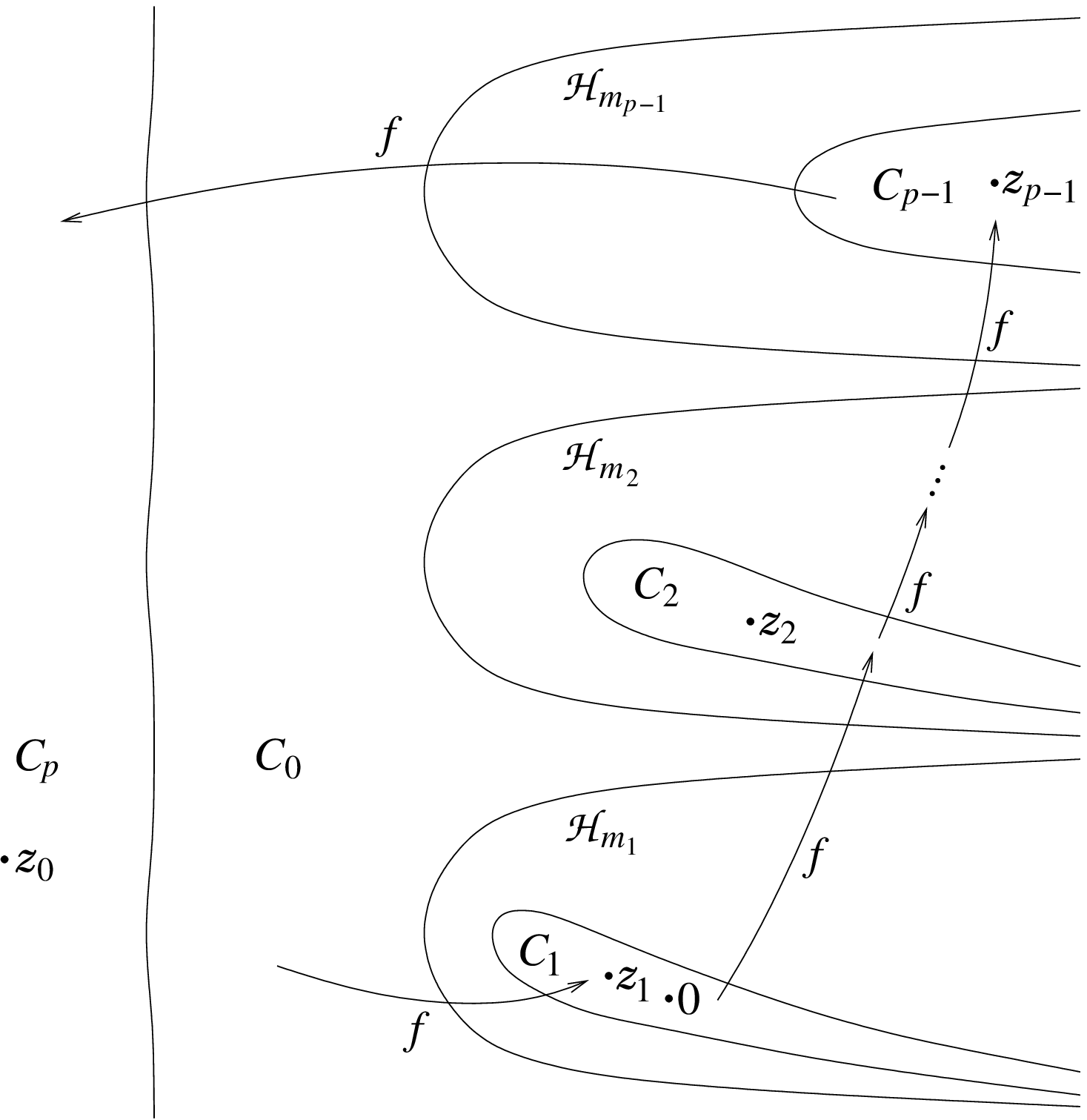}
\caption{The sets $C_j$.}
\label{fig:c}
\end{center}
\end{figure}

The following lemma is essentially stated as Proposition~2.10 in \cite{BD}.

\begin{lem} \label{lem:graphs} For $s \in \Z$, $j = 1, \ldots, p - 1$,
  we have $\bd\HH_s = \Gamma_s$, $\bd C_j = \gamma_j$, 
where $\Gamma_s, \gamma_j: \R \to \C$ are curves homeomorphic to the 
real line, such that 
\[
\Re(\Gamma_s(t)), \Re(\gamma_j(t)) \xrightarrow[t \to \pm \infty]{} +\infty.
\]
Moreover, for sufficiently large $M > 0$, in the right
  half-plane $\{z: \Re(z) \geq M\}$ the
  curves $\Gamma_s$, $\gamma_j$ form pairs of the graphs of smooth
  functions $y = H^\pm_s(x)$, $y = h^\pm_j(x)$, $x \geq M$, where $(x,
  y)$ are Euclidean coordinates on the plane and
\begin{align*}
\lim_{x \to +\infty} H^\pm_s(x) &= (2s + 1 \pm 1)\pi -
\Arg(\lambda),\\
 \lim_{x \to
  +\infty} h^\pm_{p - 1}(x) &= (2 k_{p - 1} + 1 \pm 1/2)\pi - \Arg(\lambda),\\ 
\lim_{x \to +\infty} h^\pm_j(x) &= 2\pi k_j - \Arg(\lambda) \quad \text{for
} j = 1, \ldots, p - 2
\end{align*}
for some branch of argument. Consequently, 
\[
\HH_s \cap \{z: \Re(z) \geq M\} \subset \{z: (2s - 1)\pi - \Arg (\lambda)
< \Im(z) < (2s + 3)\pi - \Arg (\lambda)\}. 
\]
\end{lem}
\begin{proof} By construction, $\gamma_{p - 1} \cap \{z: \Re(z) \geq
  M\}$ for large $M$ is a pair of curves, which are images of
  vertical half-lines under some branch of logarithm. Hence, these curves
  are the graphs of some smooth functions $h^\pm_{p - 1}(x)$, such
  that $\lim_{x \to +\infty} h^\pm_{p - 1}(x) = (2 k_{p - 1} + 1 \pm
  1/2)\pi - \Arg(\lambda)$ and 
  $|(h^\pm_{p - 1})'| < \varepsilon$ for a small $\varepsilon > 0$.  
Moreover, the
  curves $\gamma_j$, $j = p - 2, \ldots, 1$ and $\Gamma_s$ are
  successive images of $\gamma_{p - 1}$
  under inverse branches of $f$, which have the form
  $\log(z/\lambda)$ for some branch of logarithm. Note also
  that $|\Arg(z)|, |\Arg(\log'(z/\lambda))| < 
  \varepsilon$ for every $z \in 
  \bigcup_{j = 1}^{p - 1} \gamma_j$, provided
  $\Re(z) \geq M$ for large $M$. This easily implies that $\gamma_j$ and
  $\Gamma_s$ on some right-half plane are pairs of smooth curves with almost
  horizontal tangent vector and the required asymptotics. 
\end{proof}

Since $\C \setminus C_1$ is simply connected and does not
contain $0$, we can define branches of $f^{-1}$,
\[
g_s: \C \setminus C_1 \to \HH_s, \qquad s \in \Z,
\]
mapping $\C \setminus C_1$ onto $\HH_s$. Let
\[
\LL = \C \setminus \bigcup_{j = 1}^p C_j
\]
By construction, we have $\LL \subset f(\LL)$, so 
\[
g_s(\LL) \subset \LL.
\]
Hence, the collection of maps
$\{g_s|_\LL\}_{s \in\Z}$ forms an infinite conformal iterated function
system.

Since the Julia set is the complement of the basin of
attraction, for each $z \in J(f)$ all its forward
iterates are contained in $\LL \setminus C_0 \subset \bigcup_s \HH_s
$, so following \cite{BD}, we can define its itinerary $(s_0, 
s_1, \ldots)$, taking $s_n 
\in \Z$ such that $f^n(z) \in \HH_{s_n}$, $n = 0, 1, \ldots$. It 
is clear that
\[
J(f) = \bigcap_{n = 0}^\infty \bigcup_{s_0, \ldots, s_n} g_{s_0} \circ \cdots
\circ g_{s_n}(\LL),
\]
and $\bigcap_{n =  0}^\infty g_{s_0} \circ \cdots \circ
g_{s_n}(\LL)$ is precisely the set of points in the Julia set with
itinerary $(s_0, s_1, \ldots)$. 

Let
\[
\underline{k}  = (k_1, \ldots, k_{p - 1})
\]
be the block of $p - 1$ symbols, where $k_1, \ldots, k_{p - 1}$ are
integers defined in \eqref{eq:k_j}. See Figure~\ref{fig:c}. The block
$\underline{k}$ describes the trajectory of the singular value $0$ and
is
called the kneading sequence for $f$ (see \cite{BD}). 
We will write
\[
G_{\underline{k}} = g_{k_1} \circ \cdots \circ g_{k_{p - 1}}, \quad 
G_{s,\underline{k}} = g_s \circ g_{k_1} \circ \cdots \circ g_{k_{p - 1}}, \quad
G_{\underline{k},s} = g_{k_1} \circ \cdots \circ g_{k_{p - 1}} \circ g_s
\]
for $s \in \Z$.

The following lemma is a kind of folklore (it follows from
estimations of hyperbolic metric, see e.g. the proofs of
Proposition~6.1 in \cite{McM} and Proposition~2.1 in \cite{UZreal}).

\begin{lem}\label{lem:expan} There exist $c > 0$, $A > 1$, such that
  for every $n > 0$, $s_0, \ldots, s_n \in \Z$ and $z \in \LL$,
\[
|(g_{s_0} \circ \cdots \circ g_{s_n})'(z)| < cA^{-n}.
\]
\qed
\end{lem}

Now we show how the boundary of a component of the basin of attraction
is described in terms of symbolic sequences. We do not give a
precise symbolic characterisation of the boundary, but provide a
necessary and a sufficient conditions, which is enough for our purposes. To
fix notation, we consider the component $U_0$.     

Let
\begin{align*}
\Sigma_{\underline{k}} &= \{(s_0, \underline{k},
s_1, \underline{k},
\ldots): s_0, s_1, \ldots \in \Z\},\\
\Sigma'_{\underline{k}} &= \{(s_0, \underline{k},
s_1,  \underline{k},
\ldots): s_0, s_1, \ldots \in \Z \text{ and } s_n \notin \{k_1, \ldots, k_{p -
  1}\} \text{ for infinitely many } n\}.
\end{align*}

\begin{prop}\label{kody}
\[
\{z \in J(f): z \text{ has itinerary in }\Sigma'_{\underline{k}}\}
\subset \bd U_0 \subset \{z \in J(f): z
\text{ has itinerary in }\Sigma_{\underline{k}}\}.
\]
\end{prop}

\begin{proof} First, we show $\bd U_0 \subset \{z \in J(f): z
\text{ has itinerary in }\Sigma_{\underline{k}}\}$. Take $z \in
\bd U_0$ and $n \geq 0$. We need to check that $f^n(z) \in \HH_{k_j}$
if $n = lp + j$ for an integer $l$ and $j \in \{1, \ldots, p -
1\}$. It is clear that $f^{lp + j}(z) \in \bd U_j$. Moreover,
$\overline{U_j}$ must be contained in $\HH_{k_j}$, since it is a
connected subset of $\bigcup_s\HH_s$ containing $z_j \in C_j \subset
\HH_{k_j}$. This shows $f^{lp + j}(z) \in \HH_{k_j}$.

Now we show $\{z \in J(f): z \text{ has itinerary in }\Sigma'_{\underline{k}}\}
\subset \bd U_0$. Take a point $z \in J(f)$ with itinerary in
$\Sigma'_{\underline{k}}$. By definition, there are arbitrarily large
$n$ such that $f^{np}(z) \in \HH_{s_n}$ and $s_n \neq k_j$ for $j \in
\{1, \ldots, p - 1\}$. This implies that for these $n$ we have
$\HH_{s_n} \subset \LL$. Let $I_n$ be a open vertical segment in
$\HH_{s_n}$ containing $f^{np}(z)$, such that both
endpoints of $I_n$ are in $\bd \HH_{s_n}$. Note that the length
of $I_n$ is less than $2\pi$ because of \eqref{eq:H_s}. 
Let $G^n$ be the branch of
$f^{-np}$ along the trajectory of $z$, i.e. $G^n =
G_{s_0, \underline{k}} \circ 
\cdots \circ G_{s_{n - 1}, \underline{k}}$. 
Since $I_n \subset \HH_{s_n}
\subset \LL$, the branch $G^n$ is defined on $I_n$.  By
Lemma~\ref{lem:expan}, the length of $G^n(I_n)$ tends to $0$
as $n \to \infty$. Moreover, by
the definition of the sets $C_j$, each map $G_{s_l, \underline{k}}$
  maps points from $U_0$ to points from $U_0$. Hence, 
the endpoints of the curve $G^n(I_n)$ are in $U_0$. Since $z \in
G^n(I_n)$, this shows $z \in \bd U_0$.
\end{proof}

\section{Escaping points in the boundary} \label{sec:esc}

In this section we prove:

\begin{prop}\label{uciekaj}
The Hausdorff dimension of the set 
$$\{z\in J(f): z \text{ has itinerary in } \Sigma_{\underline{k}}\} \cap I(f)$$
is equal to $1$.
\end{prop}

In the proof of this proposition, it is more convenient to consider the set
\[
\tilde\Sigma_{\underline{k}} = \{(\underline{k},
s_1, \underline{k}, s_2, \ldots): s_1, s_2, \ldots \in \Z\}
\]
instead of $\Sigma_{\underline{k}}$. Note that the set $\{z\in J(f): z
\text{ has itinerary in } \Sigma_{\underline{k}}\} \cap I(f)$ is a countable
union of the images of $\{z\in J(f): z \text{ has itinerary in }
\tilde\Sigma_{\underline{k}}\} \cap I(f)$ under inverse branches of
$f$, so both sets have the same Hausdorff dimension. Hence, to prove
Proposition~\ref{uciekaj}, it is sufficient to show that 
\begin{equation}\label{eq:tilde}
\dim_H(\{z\in J(f): z \text{ has itinerary in }
\tilde\Sigma_{\underline{k}}\} \cap I(f)) = 1.
\end{equation}
Let 
\[
W_M = \{z: \Re(z) \geq M\}
\]
for a large integer $M > 0$ and
\[
A_M=\{z\in J(f):z \text{ has itinerary in }
\tilde\Sigma_{\underline{k}}\text{ and } f^n(z) \in W_M \text{ for
  all } n\ge 0\}.
\]
The assertion \eqref{eq:tilde} will follow immediately from 

\begin{prop}\label{prop:dimA_M}
For every $\delta>0$ there exists $M(\delta)>0$ such that for every
$M>M(\delta)$, 
\[
\dim_H A_M < 1+\delta.
\]
\end{prop}

The rest of this section is devoted to the proof of
Proposition~\ref{prop:dimA_M}. 

\begin{defn*}
Let
\[
Q^r_s = \LL \cap \HH_s \cap \{z: r\le \Re(z)<r+1\}
\]
for $r, s \in \Z$, $r \geq M$ and let
$\QQ^r$ be the family of all
connected components of $Q^r_{k_1}$. 
\end{defn*}

\begin{rem}\label{liczba} 
Note that if $M$ is sufficiently large, then by Lemma~\ref{lem:graphs}, 
each family $\QQ^r$ consists of at most $p$ sets.
\end{rem}

Let 
\[
\KK_0 = \bigcup_{r = M}^\infty\QQ^r
\]
and for $n \geq 0$ define inductively
\[
\KK_{n + 1} = \{G_{\underline{k}, s}(K): K \in \KK_n, s \in
\Z \text{ and } G_{\underline{k}, s}(K) \cap W_M \neq \emptyset\}.
\]
Note that if $x\in A_M$, then for every $n\ge 0$ we have $f^{np}(x)
\in \LL \cap \HH_{k_1} \cap 
W_M$, so $f^{np}(x)$ belongs to some set $K_0 \in \KK_0$. Since
\[
\LL \cap \HH_{k_1} \cap W_M = \bigcup_{K_0 \in \KK_0} K_0,
\]
this implies that the collection
$\KK_n$ is a cover of $A_M$. 

Fix $\delta > 0$. We will show that 
\begin{equation}\label{eq:diamKn}
\sup \{\diam K: K \in \KK_n\} \to 0 \text{ as } n\to\infty.
\end{equation}
and 
\begin{equation}\label{eq:K}
\sum_{K \in \KK_n}  (\diam K)^{1 +
  \delta} < 1
\end{equation}
for every $n \geq 0$, which will prove Proposition~\ref{prop:dimA_M} by
the definition of the Hausdorff measure. 

To show \eqref{eq:diamKn} and \eqref{eq:K}, we first prove two simple lemmas.

\begin{lem}\label{lem:Re}
For every sufficiently large $M$ there exists $M_1 > 0$, such that
$M_1 \to \infty$ as $M \to \infty$, and for every $z \in \LL \cap
W_M$ and $s \in \Z$,  
\[
f^j(G_{\underline{k}, s}(z)) \in W_{M_1}
\]
for every $j = 1, \ldots, p - 1$. 
\end{lem}
\begin{proof}
It it sufficient to notice that $\Re(g_s(z))  = \ln|z| - \ln|\lambda|
\geq \ln(\Re(z)) - \ln|\lambda|$.
\end{proof}

\begin{lem}\label{lem:der}
For every $z \in \LL \cap W_M$ for sufficiently large $M$, 
\[
|G_{\underline{k}, s}'(z)| < \frac{1}{\Re(z)(\pi |s| + 1)}.
\]
\end{lem}
\begin{proof} By the chain rule,
\[
|G_{\underline{k}, s}'(z)| =  |g_s'(z)||g_{k_{p - 1}}'(g_s(z))|
|(g_{k_1} \circ \cdots \circ g_{k_{p - 2}})'(g_{k_{p - 1}} \circ g_s(z))|.
\]
We have $|g_s'(z)| = 1/|z| \leq 1/\Re(z)$. 
Moreover, by Lemma~\ref{lem:graphs}, there exists a constant $c > 0$, such that 
\[
|g_{k_{p - 1}}'(g_s(z))| = \frac{1}{|g_s(z)|} \leq 
\frac{2}{(|\Re(g_s(z))|+|\Im(g_s(z))|)} <
\frac{2}{\ln M + 2\pi |s| - c} < \frac{1}{\pi |s| + 1}
\]
for large $M$. Finally, by Lemma~\ref{lem:Re}, 
\[
|(g_{k_1} \circ \cdots \circ g_{k_{p - 2}})'(g_{k_{p - 1}} \circ g_s(z))| 
= \prod_{j = 1}^{p - 2}\frac{1}{|f^j(G_{\underline{k}, s}(z))|} \leq 
\prod_{j = 1}^{p - 2}\frac{1}{\Re(f^j(G_{\underline{k}, s}(z)))} <
\frac{1}{M_1^{p - 2}} \leq 1,
\]
if $M$ is sufficiently large. 
\end{proof}

Note that by Lemma~\ref{lem:graphs}, for every
$K_0 \in \KK_0$ we have 
\begin{equation} \label{eq:diamK}
\diam K_0 < \sqrt{1 + 16\pi^2}.
\end{equation}
Moreover, by Lemma~\ref{lem:der}, for every $K \in \KK_n$,
\[
\diam G_{\underline{k}, s}(K) < \frac{\diam K}{M - \diam K},
\]
which shows \eqref{eq:diamKn} by a simple induction. Moreover, \eqref{eq:diamKn}
and Lemma~\ref{lem:der} imply immediately the following fact.

\begin{cor} \label{cor:derK}
For every $K \in \KK_n$ and every $z
\in K$, 
\[
|G_{\underline{k}, s}'(z)| < \frac{1}{(M - d)(\pi |s| + 1)}
\]
for some constant $d > 0$. \hfill\qed
\end{cor}

Now we are ready to prove $\eqref{eq:K}$. Using 
Remark~\ref{liczba}, Lemma~\ref{lem:der} and Corollary~\ref{cor:derK}
we get, for some constants $c_1, c_2 > 0$, 
\begin{multline*}
\sum_{K \in \KK_n} (\diam K)^{1 + \delta} = 
\sum_{K_0 \in \KK_0} \sum_{\substack{s_1, \ldots, s_n \in
    \Z:\\ G_{\underline{k}, s_j} \circ \cdots \circ G_{\underline{k},
      s_n}(K_0) \in \KK_{n - j + 1}\\
 \text{for } j = 1, \ldots, n}} (\diam G_{\underline{k}, s_1} \circ
\cdots \circ G_{\underline{k}, s_n}(K_0))^{1 + \delta}\\
\leq \sum_{r = M}^\infty
\sum_{K_0 \in \QQ^r} \sum_{\substack{s_1, \ldots, s_n \in
    \Z:\\ G_{\underline{k}, s_j} \circ \cdots \circ G_{\underline{k},
      s_n}(K_0) \in \KK_{n - j + 1}\\
 \text{for } j = 1, \ldots, n}} \sup_{z \in K_0} \left(|G'_{\underline{k},
  s_n}(z)| \prod_{j = 1}^{n - 1} |G'_{\underline{k},
  s_j}(G_{\underline{k}, s_{j + 1}} \circ \cdots \circ
G_{\underline{k}, s_n}(z))|\right)^{1 + \delta}\\
< p \sum_{r = M}^\infty
\sum_{s_1, \ldots, s_n \in  \Z} \left(\frac{1}{r(\pi |s_n| + 1)}
\prod_{j = 1}^{n - 1} \frac{1}{(M - d)(\pi |s_j| + 1)}\right)^{1 + \delta}\\
\leq p \sum_{r = M}^\infty \frac{1}{r^{1 + \delta}}
\left(\frac{1}{(M - d)^{1 + \delta}}\right)^{n - 1} \left(\sum_{s \in \Z}
\frac{1}{(\pi |s| 
  + 1)^{1 + \delta}}\right)^n < \frac{pc_1}{M^\delta} 
\frac{c_2^n}{(M - d)^{(1 + \delta)(n - 1)}} < 1
\end{multline*}
if $M$ is sufficiently large, which shows \eqref{eq:K} and ends the proof of 
Proposition~\ref{prop:dimA_M}. 

\section{Non-escaping points in the boundary}\label{sec:non-esc}

\begin{prop}\label{prop:nieuciekaj}
The Hausdorff dimension of the set 
\[
\{z\in J(f): z \text{ has itinerary in } \Sigma'_{\underline{k}} \text{
  and the forward orbit of $z$ under $f$ is bounded}\}
\]
is greater than $1$. 
\end{prop}

\begin{proof}
The basic idea of the proof is to construct a finite conformal Iterated
Function System with a compact Cantor repeller $X$ contained in the
considered set, such that
$f^{2p}(X)=X$, and to use classical Bowen's formula to estimate the
Hausdorff dimension of $X$. In fact, it is a modification of the proof
from \cite{BKZ}. For completeness, we present the main ideas of the proof. 

Let $R$ be a large positive real number.
Consider the square
$$Q=\left[\frac R2,\frac{3R}2\right]\times\left[\frac
  R2,\frac{3R}2\right].$$
Note that $Q \subset \LL$ for sufficiently large $R$, so all inverse branches
of the form $G_{u,\underline{k}}\circ G_{s,\underline{k}}$, $u, s \in
\Z$ are defined on $Q$. Moreover, $\diam Q = \sqrt{2}R$ and 
$\dist(Q, \bd \LL)> R/2 - D$ for some constant $D > 0$. Hence, by the
Koebe Lemma (see e.g. \cite{CG}), the distortion of these branches on
$Q$ is bounded by a constant independent of $R, u, s$.

Let 
$$Q_{u,s}=G_{u,\underline{k}}\circ G_{s,\underline{k}}(Q)$$
for $u, s \in \Z$ and
$$\G=\{(u,s)\in\Z^2:\ Q_{u,s}\subset Q\}.$$
Now define
$$X=\bigcap_{n =
  1}^\infty\bigcup_{i_1,\ldots,i_n}h_{i_1}\circ\ldots\circ h_{i_n}(Q), 
$$
where $h_{ij}=G_{u,\underline{k}}\circ G_{s,\underline{k}}$ for some $(u,s)\in\G$.
By construction,
\[
X\subset\{z\in J(f): z \text{ has itinerary in } \Sigma'_{\underline{k}} \text{
  and the forward orbit of $z$ under $f$ is bounded}\}
\] 
and $f^{2p}(X)=X$.

To prove Proposition~\ref{prop:nieuciekaj}, we will show that
\begin{equation}\label{eq:dimX}
\dim_H(X) > 1.
\end{equation} 
By Bowen's formula (see e.g. \cite{Bo, PU}), the
Hausdorff dimension of $X$ is determined as the unique $t\in\R$, such
that $P(t)=0$ where $P(t)$ is the pressure function: 
$$P(t)=\lim_{n\to\infty}\frac{1}{n}\ln\sum_{i_1,\ldots,i_n}\|(h_{i_1}
\circ\ldots\circ h_{i_n})'\|^t.$$  
The function $t\mapsto P(t)$ is
strictly decreasing, so to prove \eqref{eq:dimX} it is enough
to show that $P(1)>0$. We have
$$P(1)\ge \lim_{n\to\infty}\frac{1}{n}\ln\left( \inf_{z\in
Q}\sum_{(u,s)\in \G} |(G_{u,\underline{k}}\circ G_{s,\underline{k}})'(z)|\right )^n.
$$
Hence, using bounded distortion, to prove \eqref{eq:dimX} it is enough
to have
\begin{equation}\label{eq:P>}
\sum_{(u,s)\in \G} |(G_{u,\underline{k}}\circ G_{s,\underline{k}})'(z_R)| > C
\end{equation}
for a large constant $C > 0$, where $z_R$ is the centre of the square
$Q$, i.e. $z_R=R+iR$.  

Now we prove \eqref{eq:P>}. 
By Lemma~\ref{lem:graphs}, for every $z \in \HH_s$, $s \in \Z$ and $j
= 1, \ldots, 
p - 1$, we have
\begin{multline}\label{eq:|g|<}
|g_{k_j}(z)| \leq |\Re(g_{k_j}(z))| + |\Im(g_{k_j}(z))| \\< \ln|z| -
\ln|\lambda| + (2|s| + 1)\pi + |\Arg(\lambda)|
< \ln |z| + 2 \pi |s| + c_3
\end{multline}
for some constant $c_3 > 0$. Using this
inductively along the backward trajectory of $z_R$, we get
\begin{equation}\label{eq:Gs'}
|G_{s,\underline{k}}'(z_R)| = \frac{1}{|z_R||g_{k_{p - 1}}(z_R)| \cdots
  |g_{k_1} \circ \cdots \circ g_{k_{p - 1}}(z_R)|} >
  \frac{1}{\sqrt{2}R (\ln R + c_4)^{p - 1}}
\end{equation}
and
\begin{equation}\label{eq:Re<}
\Re(G_{s,\underline{k}}(z_R)) \leq |G_{s,\underline{k}}(z_R)| < \ln R + c_4 
\end{equation}
for some constant $c_4 > 0$. Note that the points 
\[
v_s=G_{s,\underline{k}}(z_R)
\]
for $s \in \Z$ are located on the same 
vertical line 
\[
\ell=\{z: \Re (z)=R_1\}.
\]
Let 
\[
\ell'=\ell\cap\{z: \Im (z) \geq R\}.
\]
In the same way as in the proof of Lemma~\ref{lem:Re}, it is easy to
see that $R_1$ is arbitrarily 
large for large $R$, so $\ell' \subset \LL$
and 
\begin{equation}\label{eq:dist}
\dist(z, \bd \LL) > D \quad \text{for every } z \in \ell'. 
\end{equation}
For $u\in\Z$ the set 
$G_{u,\underline{k}}(\ell')$ is a curve tending to
infinity and contained in $\HH_u$. Moreover, using \eqref{eq:Re<} and
\eqref{eq:|g|<} inductively along the backward trajectory of $\ell'$, we get
\[
\inf_{z\in G_{u,\underline{k}}(\ell')}\Re (z)< \ln R + c_4 <\frac R2, 
\] 
if $R$ is sufficiently large. Hence, each
$G_{u,\underline{k}}(\ell')$ intersects the strip $\{z: R/2 + 1 <
\Re(z) < 3R/2 - 1\}$ along a curve $\ell'_u$ of length at least $R - 2$. 
By Lemma~\ref{lem:graphs}, there are at least $c_5R$ such curves $\ell'_u$
contained in 
\[
Q' = \left[\frac{R}{2} + 1, \frac{3R}{2} - 1\right] \times  
\left[\frac{R}{2} + 1, \frac{3R}{2} - 1\right],
\]
for some constant
$c_5 > 0$. By \eqref{eq:dist} and the Koebe Lemma, the distortion of
$G_{u,\underline{k}}$ is uniformly bounded on every vertical segment of
endpoints $v_s, v_{s + 1}$, so 
\begin{equation} \label{eq:length<}
R - 2 \leq \length(\ell'_u) < c_6 \sum_{s:\, G_{u,\underline{k}}(v_s) \in
  \ell'_u}|G'_{u,\underline{k}}(v_s)|
\end{equation}
for some constant $c_6 > 0$. Moreover, $\diam Q = \sqrt{2} R$ and
$(G_{u,\underline{k}} \circ G_{s,\underline{k}})'(z) < 1/(\sqrt{2}R)$
for $z \in Q$ 
(the proof is analogous to the proof of Lemma~\ref{lem:der}), so
\[
\diam Q_{u, s} < 1.
\]
Hence, if
$G_{u,\underline{k}}(v_s) \in \ell'_u \subset Q'$, then 
$Q_{u, s} \subset Q$, so $(u, s) \in \G$. 
Using \eqref{eq:length<}, we conclude that 
$$\sum_{(u,s)\in \G} |G'_{u,\underline{k}}(v_s)| \geq \sum_{u:\,
  \ell'_u \subset Q'} \sum_{s: \,G_{u,\underline{k}}(v_s) \in \ell'_u}
|G'_{u,\underline{k}}(v_s)| > \frac{c_5}{c_6} R (R - 2).
$$
Combining this with \eqref{eq:Gs'}, we get 
\[
\sum_{(u,s)\in \G} |(G_{u,\underline{k}} \circ
G_{s,\underline{k}})'(z_R)| = \sum_{(u,s)\in
  \G}|G'_{u,\underline{k}}(v_s)| |G'_{s,\underline{k}}(z_R)| >
\frac{c_5}{\sqrt{2}c_6}\frac{R - 2}{ (\ln R + c_4)^{p - 1}},
\]
which shows \eqref{eq:P>} for sufficiently large $R$. This completes the proof.

\end{proof}

Note that it follows from the above proof that the set $X\cup f^p(X)$
is a conformal $f^p$-invariant Cantor set and
$|(f^p)'|>1$ on $X\cup f^p(X)$. Hence we immediately obtain the
following. 

\begin{prop}\label{prop:hyp}
The hyperbolic dimension of the set 
\[
\{z\in J(f): z \text{ has itinerary in } \Sigma'_{\underline{k}} \text{
  and the forward orbit of $z$ under $f$ is bounded}\}
\]
is greater than $1$. \qed
\end{prop}

\section{Conclusion} \label{sec:concl}

By Propositions~\ref{kody}, \ref{uciekaj} and \ref{prop:nieuciekaj},
the set $\bd U_0 \cap 
I(f)$ has Hausdorff dimension $1$, while $\bd U_0 \cap J_\text{bd}(f)$ has
Hausdorff dimension greater than $1$. This implies 
\[
1 < \dim_H(\bd U_0 \cap J_\text{bd}(f)) \leq \dim_H(\bd
U_0 \setminus I(f)) = \dim_H(\bd U_0).
\]
Of course, the same holds for any
component $U$ of the entire basin of attraction. This proves
the first part of Theorem~A, the lower estimate of the first part and
the second part of
Theorem~B. The remaining part is a consequence
of the second part of Theorem~B and the fact $\dim_H(J(f) 
\setminus I(f)) < 2$, which was proved in \cite{UZfiner} as
Theorem~6.1. Finally, the remark after Theorem~B
follows from Propositions~\ref{kody} and~\ref{prop:hyp}.

\end{document}